\documentclass{amsart}%
\usepackage{amsfonts}%
\usepackage{amsmath}%
\usepackage{amssymb}%
\usepackage{graphicx}
\newtheorem{theorem}{Theorem}
\theoremstyle{plain}

\newtheorem{conjecture}{Conjecture}
\newtheorem{corollary}{Corollary}

\newtheorem{lemma}{Lemma}

\numberwithin{equation}{section}
\begin{document}
\title{A Rigidity Theorem for the Hemi-Sphere}
\author{Fengbo Hang}
\address{Courant Institute, 251 Mercer Street, New York, NY 10012}
\email{fengbo@cims.nyu.edu}
\author{Xiaodong Wang}
\address{Department of Mathematics, Michigan State University, East Lansing, MI 48864}
\email{xwang@math.msu.edu}
\maketitle

\section{\bigskip Introduction}

In this paper we prove the following rigidity theorem.

\begin{theorem}
\label{main}Let $\left(  M^{n},g\right)  $ ($n\geq2$) be a compact Riemannian
manifold with nonempty boundary $\Sigma=\partial M$. Suppose

\begin{itemize}
\item \textrm{$Ric$}$\geq\left(  n-1\right)  g,$

\item $\left(  \Sigma,g|_{\Sigma}\right)  $ is isometric to the standard
sphere $\mathbb{S}^{n-1}\subset\mathbb{R}^{n}$,

\item $\Sigma$ is convex in $M$ in the sense that its second fundamental form
is nonnegative.
\end{itemize}

Then $\left(  M^{n},g\right)  $ is isometric to the hemisphere $\mathbb{S}%
_{+}^{n}\subset\mathbb{R}^{n+1}$.
\end{theorem}

\bigskip\bigskip

It may be necessary to make precise certain definitions involved here as there
are different conventions for the second fundamental form and the mean
curvature in the literature. Let $\nu$ be the outer unit normal field of
$\Sigma$ in $M$. For any $p\in\Sigma$, for any $X,Y\in T_{p}\Sigma$ the second
fundamental form is defined as%
\[
\Pi\left(  X,Y\right)  =\left\langle \nabla_{X}\nu,Y\right\rangle .
\]
The mean curvature is the trace of the second fundamental form.

Put in another way, the theorem says that for a compact manifold with
boundary, if we know that the boundary is $\mathbb{S}^{n-1}$(intrinsic
geometry on the boundary) and convex (some extrinsic geometry) then we
recognize the manifold as the hemisphere $\mathbb{S}_{+}^{n}$, provided
$\mathrm{Ric}\geq\left(  n-1\right)  g$. To put this result in a context, we
first recall the following

\begin{theorem}
\label{ball}Let $\left(  M^{n},g\right)  $ be a compact Riemannian manifold
with boundary and scalar curvature $R\geq0$. If the boundary is isometric to
$\mathbb{S}^{n-1}$ and has mean curvature $n-1$, then $\left(  M^{n},g\right)
$ is isometric to the unit ball $\overline{\mathbb{B}^{n}}\subset
\mathbb{R}^{n}$. (If $n>7$ we need to assume that $M$ is spin.)
\end{theorem}

This remarkable result is a simple corollary of the positive mass theorem:
indeed one may glue $M$ with $\mathbb{R}^{n}\backslash\mathbb{B}^{n}$ along
the boundary $\mathbb{S}^{n-1}$ to obtain an asymptotically flat manifold $N$
with nonnegative scalar curvature. Since it is actually flat near infinity the
positive mass theorem implies that $N$ is isometric to $\mathbb{R}^{n}$ and
hence $M$ is isometric to $\overline{\mathbb{B}^{n}}$ (see \cite{M, ST} for
details). There are similar rigidity results for geodesic balls in the
hyperbolic space assuming $R\geq-n\left(  n-1\right)  $ by applying the
positive mass theorem for asymptotically hyperbolic manifolds.

It is a natural question to consider the hemisphere. The following conjecture
was proposed by Min-Oo in 1995.

\begin{conjecture}
(Min-Oo) Let $\left(  M^{n},g\right)  $ be a compact Riemannian manifold with
boundary and scalar curvature $R\geq n\left(  n-1\right)  $. If the boundary
is isometric to $\mathbb{S}^{n-1}$ and totally geodesic, then $\left(
M^{n},g\right)  $ is isometric to the hemisphere $\mathbb{S}_{+}^{n}$.
\end{conjecture}

The proof of Theorem \ref{ball} does not seem to work any more: there is no
positive mass theorem providing a miraculous passage from the compact manifold
in question to a noncompact manifold. As it stands this conjecture seems
difficult. There have only been some partial results in \cite{HW} and some
recent progress in dimension three in \cite{E}. Theorem \ref{main} can be
viewed as the Ricci version of Min-Oo's conjecture. It is a strong evidence
that Min-Oo's conjecture should be true.

In dimension 2 it turns out that Theorem \ref{main} is essentially equivalent
to a result of Toponogov on the length of simple closed geodesics on a
strictly convex surface. This connection is discussed in Section 2 in which we
also present a different proof working only in dimension 2. This proof may
have some independent interest. It is also interesting to compare this two
dimensional argument, which is partly geometric and partly analytic, with the
unified proof of purely analytic nature presented in Section 3.

\textbf{Acknowledgement: }The research of F. Hang is supported by National
Science Foundation Grant DMS-0647010 and a Sloan Research Fellowship. The
research of X. Wang is supported by National Science Foundation Grant
DMS-0505645. We would like to thank Christina Sormani for valuable discussions.

\section{\bigskip The two dimensional case}

When $n=2$ we consider a compact surface $(M^{2},g)$ with boundary. The
boundary then consists of closed curves and there is no intrinsic geometry
except the lengths of these curves. The extrinsic geometry of the boundary is
given by the geodesic curvature. Therefore Theorem \ref{main} follows from the
following slightly stronger result.

\begin{theorem}
Let $(M^{2},g)$ be compact surface with boundary and the Gaussian curvature
$K\geq1.$ Suppose the geodesic curvature $k$ of the boundary $\gamma$
satisfies $k\geq c$ $\geq0$. Then $L(\gamma)\leq2\pi/\sqrt{1+c^{2}}$. Moreover
equality holds iff $(M,g)$ is isometric to a disc of radius $\cot^{-1}(c)$ in
$\mathbb{S}^{2}$.

\begin{proof}
By Gauss-Bonnet formula
\[
2\pi\chi\left(  M\right)  =\int_{M}Kd\sigma+\int_{\gamma}kds>0,
\]
where $\chi\left(  M\right)  $ is the Euler number of $M$. Therefore $M$ is
simply connected and in particular $\gamma$ has only one component. By the
Riemann mapping theorem, $(M,g)$ is conformally equivalent to the unit disc
$\overline{\mathbb{B}}\mathbb{=}\left\{  z\in\mathbb{C}:\left\vert
z\right\vert \leq1\right\}  $. Without loss of generality, we take $(M,g)$ to
be $(\overline{\mathbb{B}},g=e^{2u}|dz|^{2})$ with $u\in C^{\infty}\left(
\overline{\mathbb{B}},\mathbb{R}\right)  $. By our assumptions we have%

\[
\left\{
\begin{array}
[c]{c}%
-\Delta u\geq e^{2u}\text{ on }\overline{\mathbb{B}},\\
\frac{\partial u}{\partial r}+1\geq ce^{u}\text{ \ on }\mathbb{S}^{1}%
\end{array}
\right.
\]
Let $\underline{u}\in C^{\infty}\left(  \overline{\mathbb{B}},\mathbb{R}%
\right)  $ such that%
\[
\left\{
\begin{array}
[c]{c}%
-\Delta\underline{u}=0\text{ on }B,\\
\left.  \underline{u}\right\vert _{\mathbb{S}^{1}}=\left.  u\right\vert
_{\mathbb{S}^{1}}.
\end{array}
\right.
\]
Then $\underline{u}\leq u$ as $u$ is superharmonic. It follows from sub-sup
solution method (see, e.g., \cite[page 187-189]{SY}) that we may find a $v\in
C^{\infty}\left(  \overline{\mathbb{B}},\mathbb{R}\right)  $ with%
\[
\left\{
\begin{array}
[c]{c}%
-\Delta v=e^{2v}\text{ on }\overline{\mathbb{B}},\\
\underline{u}\leq v\leq u.
\end{array}
\right.
\]
Since $v\leq u$ and $v|_{\mathbb{S}^{1}}=u|_{\mathbb{S}^{1}}$ we have
$\frac{\partial v}{\partial\nu}\geq\frac{\partial u}{\partial\nu}$ and hence
$\left.  \frac{\partial v}{\partial\nu}\right\vert _{S^{1}}+1\geq ce^{u}$,
i.e. the boundary circle has has geodesic curvature $\geq c$. As the metric
$\left(  \overline{\mathbb{B}},e^{2v}|dz|^{2}\right)  $ has curvature $1$ and
the boundary circle is convex, it can be isometrically embedded as a domain in
$\mathbb{S}^{2}$, say $\Omega$. Denote $\sigma=\partial\Omega$ parametrized by
arclength. Notice $L\left(  \sigma\right)  =L\left(  \gamma\right)  $ as $v=u$
on the boundary $\mathbb{S}^{1}$. Because the boundary has geodesic curvature
$\geq c\geq0$, it is known that the smallest geodesic disc $D$ containing
$\Omega$ has radius at most $\cot^{-1}(c)$. Hence $L\left(  \gamma\right)
=L(\sigma)\leq2\pi/\sqrt{1+c^{2}}=L\left(  \partial D\right)  $. The equality
case follows directly from the argument.
\end{proof}
\end{theorem}

As a corollary we have the following theorem due to Toponogov.

\begin{corollary}
(Toponogov \cite{T}) Let $(M^{2},g)$ be a closed surface with Gaussian
curvature $K\geq1$. Then any simple closed geodesic in $M$ has length at most
$2\pi$. Moreover if there is one with length $2\pi$, then $M$ is isometric to
the standard sphere $\mathbb{S}^{2}$.
\end{corollary}

\begin{proof}
Suppose $\gamma$ is a simple close geodesic. We cut $M$ along $\gamma$ to
obtain two compact surfaces with the geodesic $\gamma$ as their common
boundary. The result follows from applying the previous theorem to either of
these two compact surfaces with boundary.
\end{proof}

\bigskip

Toponogov's original proof, as presented in Klingenberg \cite[page 297]{K}
uses his triangle comparison theorem. In applying the triangle comparison
theorem, which requires at least two minimizing geodesics, the difficulty is
to know how long a geodesic segment is minimizing without assuming an upper
bound for curvature. As the proof presented above, this difficulty is overcome
by using special features of two dimensional topology.

\bigskip

\section{\bigskip The proof of the main theorem}

\bigskip We now present a proof of Theorem \ref{main} which works in any
dimension $n\geq2$. We first recall the following result due to Reilly.

\begin{theorem}
(Reilly \cite{R}) Let $\left(  M^{n},g\right)  $ be a compact Riemannian
manifold with nonempty boundary $\Sigma=\partial M$. Assume that
$\mathrm{Ric}\geq\left(  n-1\right)  g$ and the mean curvature of $\Sigma$ in
$M$ is nonnegative. Then the first (Dirichlet) eigenvalue $\lambda_{1}$ of
$-\Delta$ satisfies the inequality $\lambda_{1}\geq n$. Moreover $\lambda
_{1}=n$ iff $M$ is isometric to the standard hemisphere $\mathbb{S}_{+}%
^{n}\subset\mathbb{R}^{n+1}$.
\end{theorem}

Therefore to prove Theorem \ref{main}, it suffices to show $\lambda_{1}\left(
M\right)  =n$. If this were not the case, then $\lambda_{1}\left(  M\right)
>n$. Therefore for every $f\in C^{\infty}\left(  \Sigma\right)  $ there is a
unique $u\in C^{\infty}\left(  M\right)  $ solving%
\begin{equation}
\left\{
\begin{array}
[c]{ccc}%
-\Delta u=nu & \text{on} & M,\\
u=f & \text{on} & \Sigma.
\end{array}
\right. \label{bv}%
\end{equation}
Define
\[
\phi=\left\vert \nabla u\right\vert ^{2}+u^{2}.
\]

\begin{lemma}
\label{sub}\bigskip\ $\phi$ is subharmonic, i.e. $\Delta\phi\geq0$.
\end{lemma}

\begin{proof}
Using the Bochner formula, the equation (\ref{bv}) and the assumption
\textrm{$Ric$}$\geq\left(  n-1\right)  g$,
\begin{align*}
\frac{1}{2}\Delta\phi & =\left\vert D^{2}u\right\vert ^{2}+\left\langle \nabla
u,\nabla\Delta u\right\rangle +\mathrm{Ric}(\nabla u,\nabla u)+\left\vert
\nabla u\right\vert ^{2}+u\Delta u\\
& \geq\left\vert D^{2}u\right\vert ^{2}-nu^{2}\\
& \geq\frac{\left(  \Delta u\right)  ^{2}}{n}-nu^{2}\\
& =0.
\end{align*}

\end{proof}

\bigskip

Denote $\chi=\frac{\partial u}{\partial\nu}$, the derivative on the boundary
in the outer unit normal $\nu$. By the assumption of Theorem \ref{main} there
is an isometry $F:\left(  \Sigma,g|_{\Sigma}\right)  \rightarrow
\mathbb{S}^{n-1}\subset\mathbb{R}^{n}$. In the following let $f=\sum_{i=1}%
^{n}\alpha_{i}x_{i}\circ F$, where $x_{1},\cdots,x_{n}$ are the standard
coordinate functions on $\mathbb{S}^{n-1}$ and $\alpha=\left(  \alpha
_{1},\cdots,\alpha_{n}\right)  \in\mathbb{S}^{n-1}$. We have%
\[
-\Delta_{\Sigma}f=\left(  n-1\right)  f,\text{ \ \ }\left\vert \nabla_{\Sigma
}f\right\vert ^{2}+f^{2}=1.
\]
Hence
\begin{equation}
\phi|_{\Sigma}=\left\vert \nabla_{\Sigma}f\right\vert ^{2}+\chi^{2}%
+f^{2}=1+\chi^{2}.\label{fb}%
\end{equation}
On the boundary $\Sigma$%
\[
-nf=\Delta u|_{\Sigma}=\Delta_{\Sigma}f+H\chi+D^{2}u\left(  \nu,\nu\right)
=-\left(  n-1\right)  f+H\chi+D^{2}u\left(  \nu,\nu\right)  ,
\]
whence
\begin{equation}
D^{2}u\left(  \nu,\nu\right)  +f=-H\chi.\label{d2un}%
\end{equation}

\begin{lemma}
\label{nf}\bigskip On $\Sigma$%
\[
\frac{1}{2}\frac{\partial\phi}{\partial\nu}=\left\langle \nabla_{\Sigma
}f,\nabla_{\Sigma}\chi\right\rangle -H\chi^{2}-\Pi\left(  \nabla_{\Sigma
}f,\nabla_{\Sigma}f\right)  .
\]

\end{lemma}

\begin{proof}
Indeed%
\begin{align*}
\frac{1}{2}\frac{\partial\phi}{\partial\nu}  & =D^{2}u\left(  \nabla
u,\nu\right)  +f\chi\\
& =D^{2}u\left(  \nabla_{\Sigma}u,\nu\right)  +\chi\left(  D^{2}u\left(
\nu,\nu\right)  +f\right) \\
& =D^{2}u\left(  \nabla_{\Sigma}f,\nu\right)  -H\chi^{2},
\end{align*}
here we have used (\ref{d2un}) in the last step. On the other hand%
\begin{align*}
D^{2}u\left(  \nabla_{\Sigma}f,\nu\right)   & =\left\langle \nabla
_{\nabla_{\Sigma}f}\nabla u,\nu\right\rangle \\
& =\nabla_{\Sigma}f\left\langle \nabla u,\nu\right\rangle -\left\langle \nabla
u,\nabla_{\nabla_{\Sigma}f}\nu\right\rangle \\
& =\left\langle \nabla_{\Sigma}f,\nabla_{\Sigma}\chi\right\rangle -\Pi\left(
\nabla_{\Sigma}f,\nabla_{\Sigma}f\right)  .
\end{align*}
The lemma follows.
\end{proof}

\begin{lemma}
\bigskip The function $\phi=\left\vert \nabla u\right\vert ^{2}+u^{2}$ is
constant and
\[
D^{2}u=-ug.
\]
Moreover $\chi=\frac{\partial u}{\partial\nu}$ is also constant and
$\Pi\left(  \nabla_{\Sigma}f,\nabla_{\Sigma}f\right)  \equiv0$.
\end{lemma}

\begin{proof}
Since $\phi$ is subharmonic, by the maximum principle $\phi$ achieves its
maximum on $\Sigma$, say at $p\in\Sigma$. Obviously we have
\[
\nabla_{\Sigma}\phi\left(  p\right)  =0,\text{ \ }\frac{\partial\phi}%
{\partial\nu}\left(  p\right)  \geq0.
\]
If $\frac{\partial\phi}{\partial\nu}\left(  p\right)  =0$, then $\phi$ must be
constant by the strong maximum principle and Hopf lemma (see \cite[page
34-35]{GT}). Then the proof of Lemma \ref{sub} implies $D^{2}u=-ug$. By
(\ref{fb}) $\chi$ is constant. It then follows from Lemma \ref{nf} that
$\Pi\left(  \nabla_{\Sigma}f,\nabla_{\Sigma}f\right)  \equiv0$.

Suppose $\frac{\partial\phi}{\partial\nu}\left(  p\right)  >0$. Then
$\chi\left(  p\right)  \neq0$, for otherwise it follows from (\ref{fb}) that
$\chi\equiv0$ and hence $\frac{\partial\phi}{\partial\nu}\left(  p\right)
\leq0$ by Lemma \ref{nf}, a contradiction. From (\ref{fb}) we conclude
$\nabla_{\Sigma}\chi\left(  p\right)  =0$. By Lemma \ref{nf}
\[
\frac{1}{2}\frac{\partial\phi}{\partial\nu}\left(  p\right)  =\left\langle
\nabla_{\Sigma}f,\nabla_{\Sigma}\chi\right\rangle \left(  p\right)  -H\chi
^{2}-\Pi\left(  \nabla_{\Sigma}f,\nabla_{\Sigma}f\right)  \leq0,
\]
here we have used the assumption that $\Sigma$ is convex, i.e. $\Pi\geq0$.
This contradicts with $\frac{\partial\phi}{\partial\nu}\left(  p\right)  >0$ again.
\end{proof}

\bigskip

Recall $f$ depends on a unit vector $\alpha\in\mathbb{S}^{n-1}$. To indicate
the dependence on $\alpha$ we will add subscript $\alpha$ to all the
quantities. Since $\Pi\left(  \nabla_{\Sigma}f_{\alpha},\nabla_{\Sigma
}f_{\alpha}\right)  \equiv0$ on $\Sigma$ for any $\alpha\in\mathbb{S}^{n-1}$
and $\left\{  \nabla_{\Sigma}f_{\alpha}:\alpha\in\mathbb{S}^{n-1}\right\}  $
span the tangent bundle $T\Sigma$ we conclude that $\Sigma$ is totally
geodesic, i.e. $\Pi=0$.

We now claim that we can choose $\alpha$ such that $\chi_{\alpha}\equiv0$.
Indeed, $\alpha\rightarrow\chi_{\alpha}$ is a continuous function on
$\mathbb{S}^{n-1}$. Clearly $u_{-\alpha}=-u_{\alpha}$ and hence $\chi
_{-\alpha}=-\chi_{\alpha}$. Therefore by the intermediate value theorem there
exists some $\beta\in\mathbb{S}^{n-1}$ such that $\chi_{\beta}\equiv0$. With
this particular choice $f=f_{\beta},u=u_{\beta}$ we have%
\[
\left\{
\begin{array}
[c]{c}%
D^{2}u=-ug,\\
\frac{\partial u}{\partial\nu}\equiv0.
\end{array}
\right.
\]
There is $q\in\Sigma$ such that $f\left(  q\right)  =\max f=1$. Then
$\nabla_{\Sigma}f\left(  q\right)  =0$ and hence $\nabla u\left(  q\right)
=0$ as $\frac{\partial u}{\partial\nu}\left(  q\right)  =0$. For $X\in T_{q}M$
such that $\left\langle X,\nu\left(  q\right)  \right\rangle \leq0$ let
$\gamma_{X}$ be the geodesic with $\overset{\cdot}{\gamma}_{X}\left(
0\right)  =X$. Note that $\gamma_{X}$ lies in $\Sigma$ if $X$ is tangential to
$\Sigma$ since $\Sigma$ is totally geodesic. The function $U\left(  t\right)
=u\circ\gamma_{X}\left(  t\right)  $ then satisfies the following%
\[
\left\{
\begin{array}
[c]{c}%
\overset{\cdot\cdot}{U}\left(  t\right)  =-U,\\
U\left(  0\right)  =1,\\
\overset{\cdot}{U}\left(  0\right)  =0.
\end{array}
\right.
\]
Hence $U\left(  t\right)  =\cos t$. Because $\Sigma$ is totally geodesic,
every point may be connected to $q$ by a minimizing geodesic. Using the
geodesic polar coordinates $\left(  r,\xi\right)  \in\mathbb{R}^{+}%
\times\mathbb{S}_{+}^{n-1}$at $q$ we can write%
\[
g=dr^{2}+h_{r}%
\]
where $r$ is the distance function to $q$ and $h_{r}$ is $r$-family of metrics
on $\mathbb{S}_{+}^{n-1}$ with
\[
\lim_{r\rightarrow0}r^{-2}h_{r}=h_{0},
\]
here $h_{0}$ is the standard metric on $\mathbb{S}_{+}^{n-1}$. Then $u=\cos r
$. The equation $D^{2}u=-ug$ implies%
\[
\frac{\partial h_{r}}{\partial r}=2\frac{\cos r}{\sin r}h_{r}%
\]
which can be solved to give $h_{r}=\sin^{2}rh_{0}$. It follows that $\left(
M,g\right)  $ is isometric to $\mathbb{S}_{+}^{n}$. This implies $\lambda
_{1}\left(  M\right)  =n$ and contradicts with the assumption $\lambda
_{1}\left(  M\right)  >n$. Theorem \ref{main} follows.

\end{document}